\pdfoutput=1                      % per rendere testo accessibile a studenti ipovedenti

\documentclass[10 pt,a4paper]{amsart}

\usepackage[accsupp]{axessibility}    % per rendere testo accessibile a studenti ipovedenti

\usepackage[english]{babel}
\usepackage{amsfonts, amsmath, amsthm, amssymb,amscd,indentfirst,mathrsfs,latexsym}
\usepackage{color, xcolor}
\usepackage{MnSymbol}
\usepackage{esint}
\usepackage{graphicx}
\usepackage{epstopdf}
\usepackage{verbatim}
\usepackage[all]{xy}
\usepackage{latexsym,amsfonts,amsthm}
\usepackage{times}
\usepackage{palatino}
\usepackage{marginnote}

\numberwithin{equation}{section}
\newtheorem{definition}{Definition}

\newtheorem{lemma}{Lemma}
\newtheorem{theorem}{Theorem}

\newtheorem{remark}{Remark}

\newtheorem*{question}{Question}

\theoremstyle{definition}

\begin{document}

\newcommand{\riem}{(M^m, \langle \, , \, \rangle)}
\newcommand{\Hess}{\mathrm{Hess}\, }
\newcommand{\hess}{\mathrm{hess}\, }
\newcommand{\ess}{\mathrm{ess}}
\newcommand{\longra}{\longrightarrow}
\newcommand{\eps}{\varepsilon}
\newcommand{\vol}{\mathrm{vol}}
\newcommand{\di}{\mathrm{d}}
\newcommand{\R}{\mathbb R}
\newcommand{\C}{\mathbb C}
\newcommand{\Z}{\mathbb Z}
\newcommand{\N}{\mathbb N}
\newcommand{\HH}{\mathbb H}
\newcommand{\Sph}{\mathbb S}
\newcommand{\metric}{\langle \, , \, \rangle}
\newcommand{\metricN}{( \, , \, )}
\newcommand{\lip}{\mathrm{Lip}}
\newcommand{\loc}{\mathrm{loc}}
\newcommand{\diver}{\mathrm{div}}
\newcommand{\disp}{\displaystyle}
\newcommand{\rad}{\mathrm{rad}}
\newcommand{\Ricc}{\mathrm{Ric}}
\newcommand{\Riem}{\mathrm{Riem}}
\newcommand{\Sect}{\mathrm{Sec}}
\newcommand{\sn}{\mathrm{sn}}
\newcommand{\cn}{\mathrm{cn}}
\newcommand{\ink}{\mathrm{in}}
\newcommand{\hol}{\mathrm{H\ddot{o}l}}
\newcommand{\capac}{\mathrm{cap}}
\newcommand{\ricc}{\operatorname{Ricc}}
\newcommand{\supp}{\operatorname{supp}}
\newcommand{\sgn}{\operatorname{sgn}}
\newcommand{\F}{\mathcal{F}}
\newcommand{\cut}{\mathrm{cut}}
\newcommand{\rk}{\mathrm{rk}}
\newcommand{\crit}{\mathrm{crit}}
\newcommand{\diam}{\mathrm{diam}}
\newcommand{\haus}{\mathscr{H}}
\newcommand{\gr}{\mathcal{G}}
\newcommand{\Bo}{\mathbb{B}}
\newcommand{\ra}{\rightarrow}
\newcommand{\dist}{\mathrm{dist}}
\newcommand{\II}{\mathrm{II}}
\newcommand{\DD}{\mathbb{D}}
\newcommand{\NN}{\mathbb{N}}
\newcommand{\tr}{\mathrm{Tr}}
\newcommand{\VV}{\mathcal{V}}
\newcommand{\PP}{\mathcal{P}}
\newcommand{\inj}{\mathrm{inj}}
\newcommand{\QQ}{\mathcal{Q}}
\newcommand{\ubar}{\overline{u}}
\newcommand{\tcr}{\textcolor{red}}
\newcommand{\tcb}{\textcolor{blue}}
\newcommand{\MM}{\mathscr{M}}
\newcommand{\FF}{\mathscr{F}}
\newcommand{\harm}{\mathscr{H}}
\newcommand{\set}[1]{{\left\{#1\right\}}}               
\newcommand{\pa}[1]{{\left(#1\right)}}                  
\newcommand{\sq}[1]{{\left[#1\right]}}                  
\newcommand{\abs}[1]{{\left|#1\right|}}                
\newcommand{\norm}[1]{{\left\|#1\right\|}}            
\newcommand{\BB}{\mathbb{B}}
\newcommand{\ber}{\mathscr{B}}

\title[Bernstein and half-space properties]{Bernstein and half-space properties for\\ 
minimal graphs under Ricci lower bounds}

\author{Giulio Colombo}
\address{Dipartimento di Matematica ``F. Enriques", Universit\`a degli studi di Milano, Via Saldini 50, I-20133 Milano (Italy).}
\email{giulio.colombo@unimi.it, marco.rigoli55@gmail.com}

\author{Marco Magliaro}
\address{Departamento de Matem\'atica, Universidade Federal do Cear\'a, Campus do Pici - Bloco 914, 60.455-760, Fortaleza - CE (Brazil).}
\email{marco.magliaro@gmail.com}

\author{Luciano Mari}
\address{Dipartimento di Matematica ``G. Peano", Universit\`a degli Studi di Torino, Via Carlo Alberto 10, 10123 Torino (Italy).}
\email{luciano.mari@unito.it}

\author{Marco Rigoli}

%\author{Marco Rigoli}
%\address{Dipartimento di Matematica, Universit\'a degli studi di Milano, Via Saldini 50, I-20133 Milano %(Italy)}
%\email{marco.rigoli55@gmail.com}

\thanks{L. Mari is supported by the project SNS17\_B\_MARI by the Scuola Normale Superiore.}

\date{\today}

\begin{abstract}
In this paper, we prove a new gradient estimate for minimal graphs defined on domains of a complete manifold $M$ with Ricci curvature bounded from below. This enables us to show that positive, entire minimal graphs on manifolds with non-negative Ricci curvature are constant, and that complete, parabolic manifolds with Ricci curvature bounded from below have the half-space property. We avoid the need of sectional curvature bounds on $M$ by exploiting a form of the Ahlfors-Khas'minskii duality in nonlinear potential theory.
\end{abstract}

%    \subjclass is required.
% 58J50  View Publications (2000-now) Spectral problems; spectral geometry; scattering theory
% 53C21  View Publications (1980-now) Methods of Riemannian geometry, including PDE methods; curvature restrictions


\maketitle
\tableofcontents

\section{Introduction}

Let $(M,\sigma)$ be a complete Riemannian manifold of dimension $m \ge 2$. In this paper we investigate entire minimal graphs, that is, solutions of the minimal surface equation
\begin{equation}\label{P}
\MM[u] \doteq \diver \left( \frac{D u}{\sqrt{1+|D u|^2}}\right) = 0 \qquad \text{on } \, M.
\end{equation}
A solution of \eqref{P} gives rise to a minimal graph $F : M \ra M \times \R$ via $F(x) = (x,u(x))$, where $M \times \R$ is endowed with coordinates $(x,t)$ and the product metric $\metric = \sigma + \di t^2$.\par
The existence problem for non-constant solutions to \eqref{P},  possibly subjected to restrictions like the positivity or a controlled growth at infinity of $u$, was first addressed when $M$ is the Euclidean space $\R^m$, and the efforts to solve it favoured the flourishing of Geometric Analysis. In particular, we recall the following achievements.

\begin{itemize}
\item[$(\ber 1)$] The celebrated solution to the \emph{Bernstein problem}: solutions of \eqref{P} on $\R^m$ are all affine if and only if $m \le 7$. The case $m=2$ was proved by Bernstein in \cite{bernstein} (see also \cite{emi,hopf_bern}, who fixed a gap in the original argument). Since then, various other proofs appeared, each one crucially depending on some peculiarities of the two-dimensional setting, see \cite{farina_minimal} for an account. The first method allowing for a higher-dimensional extension was given by Fleming \cite{fleming}, and since then, works of De Giorgi (\cite{dg1,dg2}, $m=3$),  Almgren (\cite{almgren}, $m=4$), Simons (\cite{simons}, $m \le 7$) and Bombieri, De Giorgi and Giusti (\cite{bdgg}, counterexamples if $m \ge 8$) settled the problem. Further examples of nontrivial, entire minimal graphs in dimension $m \ge 8$ have been constructed by Simon \cite{simon_jdg}.  
\item[$(\ber 2)$] In any dimension, positive solutions of \eqref{P} are constant. This is due to Finn \cite{finn} for $m=2$, and to Bombieri, De Giorgi and Miranda \cite{bdgm} for $m \ge 3$.
\item[$(\ber 3)$] In any dimension, solutions of \eqref{P} with at most linear growth on one side have bounded gradient \cite{bdgm}. In view of results by  Moser \cite{jmoser} and De Giorgi \cite{dg1,dg2}, they are therefore affine.
\end{itemize}

Another interesting result related to $(\ber 3)$ guarantees that solutions of \eqref{P} are affine whenever $m-7$ partial derivatives of $u$ are bounded on one side (not necessarily the same). This was obtained by Farina in \cite{far2}, see also previous works in \cite{bg,far1}.\par
On manifolds different from $\R^m$, the set of solutions of \eqref{P} may drastically change. For instance, if $M$ is the hyperbolic space $\HH^m$ of curvature $-1$, by \cite{nr} (for $m=2$) and \cite{ES_F_R} (for $m \ge 3$) equation \eqref{P} has uncountably many bounded, non-constant solutions: for every $\phi$ continuous on $\partial_\infty \HH^m$, there exists $u$ solving \eqref{P} on $\HH^m$ and approaching $\phi$ at infinity. The search for sharp conditions to guarantee the solvability of such Plateau's problem at infinity on a Cartan-Hadamard manifold, that is, on a complete, simply-connected manifold with non-positive sectional curvature, motivated various interesting works in recent years \cite{galvezrosenberg,rite_1,chr,chr2,chh,hr,chh_nonexistence}. The emerging picture, rather exhaustive, is subtle, and the interested reader can find a detailed account in \cite{esko_survey}. We only mention here that, in dimension $m \ge 3$, a pinching on the sectional curvature, $\Sect$, is needed. However, with the exception of \cite{chh_nonexistence} that will be recalled below, none of these works provide non-existence theorems in the spirit of $(\ber 1), (\ber 2), (\ber 3)$, leading us to formulate the following

% The interested reader may consult \cite[Thm. 1.5]{chr} and \cite{chr,chh}, as well as \cite{hr} for related interesting non-existence results. This %naturally motivates the following 

\begin{question} 
Which geometric conditions on $M$ can guarantee the validity of results analogous to $(\ber 1), (\ber 2), (\ber 3)$?
\end{question}

To date, a thorough answer seems still far from being reached. In view of the structure theory initiated by Cheeger and Colding in  \cite{cc_almost_rigidity,cc_1,cc_2,cc_3}, it is reasonable to hope that results similar to those in $(\ber 1),(\ber 2), (\ber 3)$ might be obtainable on manifolds with non negative Ricci curvature, $\Ricc \ge 0$. This is reinforced by the analogy with the behaviour of harmonic functions on $M$: indeed, positive harmonic functions on a manifold with $\Ricc \ge 0$ are constant \cite{yau, chengyau}, and harmonic functions with linear growth provide splitting directions for the tangent cones at infinity of $M$, and for $M$ itself provided that $\Ricc \ge 0$ is strengthened to $\Sect \ge 0$ \cite{ccm}. The parallelism is made even more evident by considering that \eqref{P} rewrites as the harmonicity of $u$ on the graph with its induced metric.\par
Also, the class of manifolds with $\Sect \ge 0$ is natural to investigate and, indeed, problem $(\ber 1)$, for which we are not aware of any results in dimension $m \ge 3$ besides those in $\R^m$, seems very hard even in this setting. On the other hand, assuming suitable lower bounds on the sectional curvature problems $(\ber 2)$ and $(\ber 3)$ are more treatable, and in recent years a few works have appeared: 
\begin{itemize}
\item[-] $(\ber 2)$ was shown to hold on manifolds with $\Ricc \ge 0$ and $\Sect \ge -\kappa^2$, for some $\kappa \in \R$, by work of Rosenberg, Schulze and Spruck \cite{rosenbergschulzespruck}. Moreover, in \cite{chh_nonexistence} Casteras, Heinonen and Holopainen proved the constancy of positive minimal graphs with at most linear growth on manifolds with asymptotically non-negative sectional curvature and only one end.
\item[-] $(\ber 3)$ was studied by Ding, Jost and Xin in \cite{dingjostxin,dingjostxin2}: there, in particular, the authors showed that $\R^m$ is the only manifold satisfying   
\[
\left\{ \begin{array}{l}
\Ricc \ge 0, \qquad \disp \lim_{r \ra \infty} \frac{|B_r|}{r^m} > 0 \\[0.3cm]
\text{the curvature tensor decays quadratically}
\end{array}\right.
\]
and supporting a non-constant minimal graph with at most linear growth on one side.
\end{itemize}

Hereafter, $B_r$ will denote the geodesic ball of radius $r$ in $(M,\sigma)$ centered at a fixed origin, and $|B_r|$ its Riemannian volume.

In any of \cite{rosenbergschulzespruck,chh_nonexistence,dingjostxin,dingjostxin2}, sectional curvature bounds are used at various technical stages. We briefly explain the main difficulty in removing them. A key tool in proving the results are local pointwise or integral estimates for the function 
	\[
	\sqrt{1+|D u|^2},
	\]
which in view of the Jacobi equation solves 
	\[
	\Delta_g W = \left(\|\II\|^2 + \frac{\Ricc(Du,Du)}{W^2}\right) W + 2 \frac{\|\nabla W\|^2}{W},
	\]
where $\|\cdot\|, \nabla, \Delta_g$ are the norm, gradient and Laplacian of the induced metric $g = F^*\metric$ on the graph of $u$, naturally identifiable with $M$ itself, and $\II$ is the second fundamental form. This is the analogue of Bochner's formula for harmonic functions on $M$, and suggests that a Ricci lower bound qualifies as the natural assumption to be put on $M$. However, $\Delta_g \phi = g^{ij}\phi_{ij}$, with $g^{ij}$ being the inverse of $g$ and $\phi_{ij} = D^2 \phi(e_i,e_j)$. The eigenvalues of $g^{ij}$ with respect to the background metric $\sigma$ are $1$ and $W^{-2}$, so $\Delta_g$ ceases to be uniformly elliptic on $M$ if $|Du| \to \infty$. A Ricci lower bound, which by comparison theory can only control the Laplacian $\sigma^{ij} r_{ij}$ of the distance to a fixed point, seems therefore unable to provide an efficient way to localize estimates for $W$ via the distance function, as customarily done, and calls for new ideas.\par
Our purpose in the present paper is to study $(\ber 2)$ on manifolds with $\Ricc \ge 0$. More generally, we address both Bernstein and half-space properties for $M$, according to the following definition in \cite{rosenbergschulzespruck}:

\begin{definition} 
\emph{We say that
	\begin{itemize}
	\item[-] $M$ has the \emph{Bernstein property} if it satisfies $(\ber 2)$: positive, entire minimal graphs over $M$ are constant; 		\item[-] $M$ has the \emph{half-space property} if each properly immersed, minimal hypersurface $S \ra M \times \R$ that lies on some upper half-space $\{t > t_0\}$ is a slice, that is, $S = \{t=t_1\}$ for some constant $t_1$.
	\end{itemize}
}
\end{definition}

In \cite{rosenbergschulzespruck}, Rosenberg, Schulze and Spruck proved the following  

\begin{theorem}[\cite{rosenbergschulzespruck}]\label{teo_rss}
Let $M$ be a complete manifold of dimension $m \ge 2$. 
\begin{itemize}
\item[$(i)$] If $\Ricc \ge 0$ and $\Sect \ge -\kappa^2$ for some constant $\kappa>0$, then $M$ has the Bernstein property.
\item[$(ii)$] If $M$ is parabolic and $|\Sect| \le \kappa^2$ for some constant $\kappa> 0$, then $M$ has the half-space property.
\end{itemize}
\end{theorem}

\begin{remark}\label{rem_parabolic}
\emph{We recall that a manifold $M$ is said to be parabolic if every function $u \in C(M) \cap W^{1,2}_\loc(M)$ solving 
	\[
	\Delta u \ge 0 \quad \text{weakly on } \, M, \qquad \sup_M u < \infty 
	\]
	is constant. By classical works (see the account in \cite{grigoryan}), a sufficient condition for parabolicity is
	\begin{equation}\label{bound_liyau}
	\int^\infty \frac{s \di s}{|B_s|} = \infty, 
	\end{equation}
which is also necessary on manifolds with $\Ricc \ge 0$ (cf. \cite{liyau_acta}). 
%However, if $\Ricc \ge - \kappa^2$ and $\kappa>0$, \eqref{bound_liyau} might be far from being necessary. 
}
\end{remark}	

Our main result removes all sectional curvature requirements in Theorem \ref{teo_rss}. It seems reasonable to conjecture that the Ricci lower bounds in the next theorem are sharp.
	
\begin{theorem}\label{mainteo_tiporss}
Let $M$ be a complete manifold of dimension $m \ge 2$. 
\begin{itemize}
\item[$(i)$] If $\Ricc \ge 0$, then $M$ has the Bernstein property $(\ber 2)$.
\item[$(ii)$] If $M$ is parabolic and $\Ricc \ge -(m-1)\kappa^2$ for some constant $\kappa> 0$, then $M$ has the half-space property.
\end{itemize}
\end{theorem}

Recently, Ding in \cite{ding} proved $(\ber 2)$ on manifolds with $\Ricc \ge 0$ by means of an entirely different method. It is interesting to compare his work, which appeared the same week as ours on arXiv, to the approach herein: on the one hand, his technique is very robust and applies, quite more generally, to manifolds possessing a doubling and a weak Neumann-Poincar\'e inequalities. On the other hand, it is restricted to entire graphs, and seems not applicable to show $(ii)$. In fact, the constancy of $u$ is obtained from H\"older bounds rather than gradient estimates.  To our knowledge, leaving aside the case of manifolds $M$ with slow volume growth (considered in Remark \ref{rem_slow} below), our and Ding's results are the first that guarantee a control on an entire minimal graph under just a Ricci bound on $M$. 

\begin{remark}[\textbf{Manifolds with slow volume growth}]\label{rem_slow}
\emph{If $M$ is complete and the volume of its geodesic balls $B_r$ centered at a fixed origin satisfies \eqref{bound_liyau}, then it is easy to prove that $M$ has both the Bernstein and the half-space properties. Indeed, an argument in \cite{liwang_mini} shows that the area of an entire minimal graph $\Sigma$ inside an extrinsic ball $\BB_r \subset M \times \R$ satisfies
	\[
	|\Sigma \cap \BB_r| \le |B_r| + \frac{1}{2} |B_{2r}\backslash B_r| \le 2|B_{2r}|.
	\]
In particular, the volume of the geodesic ball $B_r^g$ in the graph metric $g$ on $\Sigma$ satisfies
	\begin{equation}\label{eq_para_1}
	\int^{\infty} \frac{s \di s}{|B^g_s|} = \infty,
	\end{equation}
and thus $\Sigma$ is a parabolic manifold by Remark \ref{rem_parabolic}. Hence, having observed that $u$ is harmonic in the graph metric, if $u>0$ on $\Sigma$ then necessarily $u$ is constant. Regarding the half-space property, the proof in \cite{liwang_mini} can be extended to show that every entire graph that is minimal outside of a compact set satisfies the bound $|\Sigma \cap \BB_r| \le C|B_{2r}|$, for some constant $C>0$. In particular, referring to Section \ref{sec_Thm2}, the inequality holds for the minimal graph $\Sigma$ that we shall construct on $M \backslash B_\eps$ in the proof of $(ii)$, Theorem \ref{mainteo_tiporss}, once the graph function $u$ is smoothly extended on $B_\eps$ in such a way that $B_\eps \equiv \{u> \delta\}$. It follows that $\bar u = \min\{ u, \delta\}$ is superharmonic and bounded on $\Sigma$, hence constant because of \eqref{eq_para_1}. This suffices to conclude the half-space property.
}
\end{remark}

Theorem \ref{mainteo_tiporss} depends on a new gradient estimate for the mean curvature operator, which we now introduce. Properties $(\ber 2)$ and $(\ber 3)$ on $\R^m$ are based on the following inequality for a minimal graph that is positive in a ball $B_{r}(x)$, obtained in \cite{bdgm} (see also \cite{trudi} for a different proof):
	\[
	|Du(x)| \le c_1\exp\left\{ c_2 \frac{u(x)}{r} \right\}, 
	\]
for some constants $c_j = c_j(m)$. For $m=2$, the necessity of a right-hand side exponentially growing in $u$ was previously observed by Finn in \cite{finn2}, where more precise estimates were found: namely, he showed that the constant $\pi/2$ in the two-dimensional estimate
	\begin{equation}\label{gradm2}
	\sqrt{1 + |Du(x)|^2} \le \exp \left\{ \frac{\pi}{2} \frac{u(x)}{r} \right\}
	\end{equation}
is sharp. The proof of \eqref{gradm2} follows by combining works in \cite{finn,finn2,js,serrin,serrin_2}, cf. also \cite{nitsche}. For \emph{entire} minimal graphs in dimension $m \ge 3$, it is worth to recall that Bombieri and Giusti in \cite{bg} were able to obtain an upper bound for $|Du(x)|$ which is polynomial, not exponential, in terms of $\sup_{B_r}|u|/r$.

Similarly, in Theorem \ref{teo_rss} the Bernstein property is a consequence of a local gradient bound. However, the approach in \cite{rosenbergschulzespruck} is different from that in \cite{bdgm} and relies on the classical Korevaar's estimates (cf. \cite{korevaar}), a method that parallels Cheng's and Yau's one for harmonic functions (with some notable differences, cf. Section \ref{sec_proof}). Under the assumptions $\Ricc \ge 0$ and $\Sect \ge - \kappa^2$ on $B_r(x)$, the authors in \cite{rosenbergschulzespruck} obtained that any  positive solution $u : B_r(x) \ra \R^+$ of \eqref{P} shall satisfy
	\[
	|Du(x)| \le c_1\exp\left\{ c_2[1 + \kappa r \coth(\kappa r)] \frac{u^2(x)}{r^2} \right\}, 
	\]
for some $c_j=c_j(m)$. 

Differently from the aforementioned results, our gradient estimate will not be obtained from an effective local one, but rather it takes the form of a maximum principle at infinity, in the spirit of \cite{prsmemoirs,bmpr}.  

\begin{theorem}\label{teo_graph_soloricci_intro}
Let $M$ be a complete manifold of dimension $m \ge 2$ satisfying 
	\[
	\Ricc \ge - (m-1)\kappa^2, 
	\]
for some constant $\kappa \ge 0$. Let $\Omega \subset M$ be an open subset and let $u \in C^\infty(\Omega)$ be a positive solution of $\MM[u] = 0$ on $\Omega$. If either
	\begin{itemize}
	\item[$(i)$] $\Omega$ has locally Lipschitz boundary %finite perimeter
	and 
	\[
	\liminf_{r \ra \infty} \frac{\log |\partial \Omega \cap B_r|}{r^2} < \infty, \qquad \text{or}
	\]
	\item[$(ii)$] $u \in C(\overline \Omega)$ and is constant on $\partial \Omega$,
	\end{itemize}
	then
	\begin{equation}\label{eq_bellissima}
	\frac{\sqrt{1+|Du|^2}}{e^{\kappa\sqrt{m-1}u}} \leq  \max\left\{ 1, \limsup_{x \ra \partial \Omega} \frac{\sqrt{1+|Du(x)|^2}}{e^{\kappa\sqrt{m-1}u(x)}}\right\} \qquad \text{on } \, \Omega.
	\end{equation}
	In particular, if $\Omega = M$, 
	\begin{equation}\label{eq_entire}
	\sqrt{1+|Du|^2} \le e^{\kappa\sqrt{m-1}u} \qquad \text{on } \, M.
	\end{equation}
If equality holds in \eqref{eq_entire} at some point, then $\kappa = 0$ and $u$ is constant. 	
\end{theorem}

\begin{remark}
\emph{For $\Omega \subset M$ open and $v : \Omega \ra \R$, we agree to write 
	\[
	\limsup_{x \to \partial \Omega} v \doteq \inf_{V \text{ open, } \overline{V} \subset \Omega} \left( \sup_{\Omega \backslash \overline{V}} v \right)
	\]
where $\overline{V}$ is the closure of $V$ in $M$. 
}
\end{remark}

\begin{remark}
\emph{It is interesting to compare Theorem \ref{teo_graph_soloricci_intro} with the corresponding result for  harmonic functions. By \cite[Thm. 2.21]{maririgolisetti_mono}, if $\Ricc \ge -(m-1)\kappa^2$ on $M$, a positive solution of $\Delta u =0$ on a domain $\Omega \subset M$ enjoys the bound
	\[
	\frac{|D u|}{(m-1)\kappa u} \le \max \left\{ 1, \limsup_{x \ra \partial \Omega} \frac{|D u(x)|}{(m-1)\kappa u(x)} \right\} \qquad \text{on } \, \Omega. 
	\]
In particular, entire solutions satisfy
	\[
	|Du| \le (m-1) \kappa u \qquad \text{on } \, M,
	\]
an estimate originally due to Cheng and Yau \cite{yau,chengyau}, and Li and Wang \cite{liwang}. Observe that, given any manifold $(T^{m-1}, \metric)$ with $\Ricc \ge 0$, the warped product
	\[
	M = T \times_{e^{\kappa r}} \R \qquad \text{with metric } \, \di r^2 + e^{2\kappa r} \metric
	\]
satisfies $\Ricc \ge -(m-1)\kappa^2$ and the function 
	\[
	u(r,y) = e^{-(m-1)\kappa r}
	\]
	is harmonic with $|Du| \equiv (m-1)\kappa u$. On the contrary, for minimal graphs we didn't succeed in providing non-constant examples of $u$ attaining the equality in \eqref{eq_entire} at some point, the reason being the difficulty to exploit the refined Kato inequality in the setting of minimal graphs. 
}
\end{remark}

We briefly comment on the proof of Theorem \ref{teo_graph_soloricci_intro}. Denote with $(\Sigma,g)$ the graph of $u$. The approach follows somehow that in \cite{rosenbergschulzespruck}: there, localization is performed by using the distance function $r$ from a fixed origin in the base manifold $M$, and as such it calls for estimating the Laplacian
	\[
	\Delta_g r = g^{ij} r_{ij}
	\]
in the graph metric $g$. As said, the lack of ellipticity of $\Delta_g$ with respect to the background metric $\sigma$ makes a Ricci lower bound unable to control $g^{ij}r_{ij}$. The key step in our proof, say for $\Omega = M$, is to replace $r$ with a function $\varrho$ that solves in the graph metric 
\begin{equation}\label{array_varrho_intro}
\left\{\begin{array}{l}
\varrho \ge 0, \qquad \varrho(x) \ra \infty \ \text{ as } \, x \ra \infty \\[0.2cm]
\Delta_g \varrho \le \eps - \|\nabla \varrho\|^2
\end{array}\right.
\end{equation}
for small enough $\eps>0$. The limit in the first line is equivalent to saying that $\varrho$ is an exhaustion, namely, that it has relatively compact sublevel sets. The construction of $\varrho$ relies on a duality principle recently discovered in \cite{marivaltorta,maripessoa,maripessoa_2}, called the \emph{Ahlfors-Khas'minskii duality} (AK-duality for short). Roughly speaking, the principle establishes that, for a large class of operators $\mathscr{F}$ including linear, quasilinear and fully nonlinear ones of geometric interest, a Liouville property for solutions $u$ of $\mathscr{F}[u] \ge 0$ that satisfy $\sup_M u < \infty$ is \emph{equivalent} to the existence of exhaustions $w$ satisfying $\mathscr{F}[w] \le 0$. In our case, the construction of $\varrho$ can be split into the following steps.
\begin{itemize}
\item[(a)] Since the graph $\Sigma$ is minimal, its geodesic balls have a mild growth at infinity (a standard calibration argument). In particular, in our assumptions
	\[
	\liminf_{r \ra \infty} \frac{\log |B_r|}{r^2} < \infty.
	\]
The latter is, by \cite{prsmemoirs}, a sufficient condition for $\Sigma$ to be a stochastically complete manifold, namely, for the Brownian motion on $M$ to be non-explosive (cf. \cite{grigoryan}).
\item[(b)] By \cite{prsmemoirs}, a stochastically complete manifold satisfies the weak maximum principle at infinity (in the sense of Pigola-Rigoli-Setti, \cite{prsmemoirs}), and in fact, the two properties are equivalent. In particular, the only entire solution of $\Delta_g u \ge \eps u$ that is non-negative and bounded from above is $u \equiv 0$ (this Liouville theorem has also been stated in \cite{grigoryan}).
\item[(c)] By the AK-duality, one can construct (many) exhaustions $w$ solving $\Delta_g w \le \eps w$ on $\Sigma$. 
\item[(d)] After some manipulation, $\rho = \log w$ solves \eqref{array_varrho_intro}.
\end{itemize} 

If $\partial \Omega$ is non-empty, we will localize the estimates on an open subset $\Omega' \subset \Omega$, but the argument goes, in principle, as in the boundaryless case. All that is required is the possibility to embed isometrically $\Omega'$ into an open set of a stochastically complete manifold, and this is the point where we need conditions $(i)$ or $(ii)$ in Theorem \ref{teo_graph_soloricci_intro}.\par

\vspace{0.2cm}	
	
The Bernstein property has the following corollary for prescribed mean curvature equations, notably including the capillarity equation, that is, the case $f(u) = cu$ with constant $c>0$. 

\begin{theorem}\label{teo_tkachev}
Let $(M,\sigma)$ be a complete manifold with $\Ricc \ge 0$. Fix an origin $o$ and let $r(x) = \mathrm{dist}(x,o)$. Then, for every $f \in C(\R)$ and $0 < b \in C(M)$ satisfying
	\[
	\begin{array}{lcl}
	b \ge C(1 + r)^{-\mu} \qquad \text{on } \, M, \\[0.2cm]
	f \not \equiv 0 \qquad \text{and non-decreasing on } \, \R, 
	\end{array}
	\]
for some constants $C>0$ and $\mu < 1$, the only solutions of 
	\begin{equation}\label{eq_nonhom}
	\diver \left( \frac{Du}{\sqrt{1+|Du|^2}} \right) = b(x) f(u)
	\end{equation}
are constants $c \in \R$ with $f(c) = 0$.	
\end{theorem}

\begin{remark}
\emph{As the example of affine functions on $\R^m$ shows, the requirement $f \not \equiv 0$ is essential.	
}
\end{remark}

On $\R^m$ and with $b \equiv 1$, the above result was first obtained by Tkachev \cite{tkachev}, see also Thm. 10.4 of Farina's survey \cite{farina_minimal}. Extensions were later found by Naito and Usami \cite[Thm. 1]{nu}, Serrin \cite[Thm. 3]{Serrin_4}, and Pucci and Serrin \cite[Thm. 8.1.3]{pucciserrin}, still in Euclidean space. On manifolds, an up-to-date account can be found in Sections 7.3, 8.2 and 10.6 in \cite{bmpr}. In particular, we stress the Liouville Theorems 2.29 and 2.32 there for equation \eqref{eq_nonhom}, that however are not sufficient to conclude the constancy of $u$ under the assumptions in Theorem \ref{teo_tkachev}.
			
\vspace{0.5cm}

\noindent \textbf{Acknowledgements.} We would like to thank Jorge H.S. de Lira, for fruitful conversations at various stages of this paper. We also thank I. Holopainen for useful comments after a first version of this work has been posted on arXiv.

\section{Preliminaries}

Let $M \times \R$ be given the product metric $\metric = \sigma + \di t^2$, with $\sigma$ the metric on $M$. Given $u \in C^2(M)$, we introduce the graph immersion $F : M \ra M \times \R$, $F(x) = (x,u(x))$, and we let $\Sigma$ be the graph of $F$, namely, $\Sigma$ is the manifold $M$ endowed with the pull-back metric $g = F^*\metric$. Norms with respect to $\sigma$ and $g$ will be denoted, respectively, by $|\cdot |$ and $\| \cdot \|$. If $\sigma_{ij}$ are the components of $\sigma$ in a local frame $\{e_i\}$ with dual coframe $\{\theta^i\}$, $1\le i,j \le m$, then 
 \begin{equation} \label{metric_graph}
	g_{ij}= \sigma_{ij}+u_iu_j, \qquad g^{ij} = \sigma^{ij} - \frac{u^iu^j}{W^2},
\end{equation}
where $\di u = u_i \theta^i$, $u^i = \sigma^{ij}u_j$ and 
$$
W = \sqrt{1+|Du|^2}.
$$ 
Also, the vectors $e_i$ are mapped to $E_i = e_i + u_i \partial_t$. We denote by $D,\nabla$ the connections of, respectively, $\sigma$ and $g$, with Christoffel symbols, respectively, $\gamma^k_{ij}$ and $\Gamma_{ij}^k$. A computation shows that 
$$
\Gamma^k_{ij} - \gamma^k_{ij} = \frac{u^k u_{ij}}{W^2},
$$
where, for a function $\phi: M \ra \R$, $\phi_{ij}$ are the components of the Hessian $D^2\phi$ in the metric $\sigma$. Since the upward pointing unit normal vector to the graph is
$$
\nu = \frac{\partial_t - u^i e_i}{W},
$$
a differentiation yields that the second fundamental form $\II$ in the $\nu$ direction has components 
\begin{equation}\label{seconffund}
\mathrm{II}_{ij} = \frac{u_{ij}}{W}.
\end{equation}
The mean curvature in the direction of $\nu$ is therefore
\begin{equation}\label{meancurv}
H = g^{ij}\mathrm{II}_{ij} = \diver \left( \frac{Du}{W}\right).
\end{equation}
Moreover, for every $\phi : M \to \R$, the components of $\nabla^2 \phi$ and $\Delta_g \phi$ in the graph metric $g$ can be written as
\begin{equation}\label{hessian_on_graph}
\left\{ \begin{array}{l}
\nabla^2_{ij} \phi = \phi_{ij} - \phi_ku^k \frac{u_{ij}}{W^2} \\[0.2cm]
\Delta_g \phi = g^{ij}\phi_{ij} - \phi_ku^k \frac{H}{W} = g^{ij}\phi_{ij},
\end{array}\right.
\end{equation}
where we used the minimality of $\Sigma$. In particular, 
\begin{equation}\label{hessian_on_graph_u}
\nabla^2_{ij} u = \frac{u_{ij}}{W^2} = \frac{\II_{ij}}{W}, \qquad \Delta_g u = \frac{H}{W} = 0.
\end{equation}
Since $\partial_t$ is a Killing field, the function $1/W = \langle \nu, \partial_t \rangle$ solves the Jacobi equation
\[
\Delta_g \left( \frac{1}{W} \right) + \Big( \|\II\|^2 + \overline{\Ricc}(\nu,\nu)\Big) \frac{1}{W} = 0,
\]
with $\overline{\Ricc}$ the Ricci curvature of $\bar M$. Using
\begin{equation}\label{nablau2}
\|\nabla u\|^2 = g^{ij} u_i u_j = \frac{W^2-1}{W^2},
\end{equation}
and 
	\[
	\overline{\Ricc}(\nu,\nu) = \Ricc \left(\frac{Du}{W},\frac{Du}{W}\right) \ge - \kappa^2(m-1) \frac{W^2-1}{W^2} 
	\]
we deduce 
	\begin{equation}\label{jacobi_forW}
	\Delta_g W \ge \left(\|\II\|^2 - (m-1)\kappa^2 \frac{W^2-1}{W^2}\right) W + 2 \frac{\|\nabla W\|^2}{W}.
	\end{equation}

\section{Proof of Theorem \ref{teo_graph_soloricci_intro}}\label{sec_proof}

We need some technical lemmas. Having fixed an origin $o \in (M,\sigma)$, let $r(x)$ denote the distance from $o$ and let $B_r$ be the geodesic ball of radius $r$ centered at $o$. Hereafter, we shall consider measures with respect to metrics $a$ different from $\sigma$, and we shall therefore denote with $B_r^a$ a geodesic ball in the metric $a$, and by $| \cdot |_a$ the Hausdorff measure induced by $a$. For notational convenience we do not specify the dimension, for instance, if $\Omega \subset M$ is an open subset, the notation $|\Omega|_a$ will denote its $m$-dimensional volume, while $|\partial \Omega|_a$ the $(m-1)$-dimensional Hausdorff measure.\par

We first need some lemmas about the possibility of isometrically including graphs with bounded mean curvature into complete, boundaryless manifolds with the same dimension and good properties.

\begin{lemma}\label{lem_stochest}
Let $(\Sigma,g)$ be a connected Riemannian manifold of dimension $m$. Let $\Sigma' \varsubsetneqq \Sigma$ be a connected open subset with smooth boundary whose closure $\overline{\Sigma'}$ in $\Sigma$ satisfies:
\begin{itemize}
%\item[(i)] $\overline{\Sigma'} \subset \Sigma$, 
%\item[(ii)] $\partial \Sigma'$ smooth and complete in the metric induced from $g$,
%\item[(iii)] every bounded subset of $\Sigma'$ has compact closure in $\Sigma$. %$\overline{\Sigma'} \cap B^g_r$ compact for every $r$, where $B^g_r$ is the ball of radius $r$ in $(\Sigma, g)$ centered at some $o \in \Sigma'$.
\item[$(\wp)$] bounded subsets of $(\overline{\Sigma'},g)$ have compact closure in $\Sigma$.
\end{itemize}
Then, $(\overline{\Sigma'},g)$ can be isometrically embedded into a complete, connected $m$-dimensional Riemannian manifold $(N,h)$ in such a way that the geodesic balls $B^h_r$ centered at $o$ in $(N,h)$ satisfy 
$$
|B^h_r|_h \le 2|\Sigma' \cap B^g_{4r}|_g + 5 \pi 
$$
for every $r \ge R_0 \doteq 2 \mathrm{dist}_g(o, \partial \Sigma') + 2$. Here, $\mathrm{dist}_g$ is the distance function in $(\Sigma,g)$ and $B^g_{4r}$ is the metric ball in $(\Sigma,g)$ of radius $4r$.
\end{lemma}

\begin{remark}\label{rem_notbad}
	\emph{Here, as usual, we say that a subset $\Sigma' \subset \Sigma$ has smooth boundary if its closure $\overline{\Sigma'}$ in $\Sigma$ is a smooth manifold with boundary, and if, for each $y \in \partial \Sigma'$, there exists a chart $(U,\varphi) \subset \Sigma$ with $U$ diffeomorphic to $\R^m$ and such that $\varphi(y) = 0$ and $\varphi(\overline{\Sigma'})$ is a closed half-space.}
\end{remark}

\begin{remark}\label{rem_comple}
\emph{Observe that $(\wp)$ is equivalent to ask that bounded subsets of $(\overline{\Sigma'},g)$ have compact closure in $\overline{\Sigma'}$. Hence, denoting with $\mathrm{dist}_g'$ the distance induced by the metric $g$ on $\overline{\Sigma'}$, by the Heine-Borel Theorem $(\wp)$ can also be rephrased as the completeness of $(\overline{\Sigma'}, \mathrm{dist}_g')$.}
\end{remark}

\begin{proof}
We construct a suitable double of $(\Sigma',g)$. Denote for convenience with $(S,g_0)$ the boundary $\partial \Sigma'$ of $\overline{\Sigma'}$ with the induced metric. First, observe that $(\wp)$ ensures the completeness of $(S,g_0)$. Indeed, a Cauchy sequence $\{y_k\} \subset (S,g_0)$ is also Cauchy in $\overline{\Sigma'}$ with the distance $\mathrm{dist}_{g}'$ induced from $g$. Hence, by $(\wp)$, $y_k \to y \in \overline{\Sigma'}$. Remark \ref{rem_notbad} guarantees that $S \hookrightarrow \overline{\Sigma'}$ is an embedding with closed image, so $y\in\partial \Sigma'$ and $y_k \to y$ also in $(S,g_0)$.

Let $\nu$ be the exterior normal of $(\Sigma',g) \hookrightarrow (\Sigma,g)$, let $\rho$ be the intrinsic distance from an origin $\bar o \in S$ with respect to the metric $g_0$, and denote with $B_r^S$ the corresponding geodesic balls. Let $t_0 : [0,\infty) \ra (0,\frac{1}{2})$ be a continuous function satisfying the following assumptions:
\begin{itemize}
\item[($p_1$)] The normal exponential map $\Psi(y,r) = \exp_y(r \nu(y))$ is a diffeomorphism in the tubular neighbourhood
$$
\mathscr{D} \doteq \Big\{ (y,r) \in S \times [0, \infty) \ : \ r < 2t_0(\rho(y)) \Big\},
$$
with $\Psi(\mathscr{D}) \subset \Sigma$.
\item[($p_2$)] Writing the pull-back metric on $\mathscr{D}$ as $g = \di r^2 + g_r(y)$, and letting $b = \di r^2 + g_0(y)$ denote the product metric on $\mathscr{D}$, $t_0$ satisfies the pointwise inequality
$$
t_0(\tau) < \frac{1}{(1+\tau^2)|\partial B_{\tau}^S|_{g_0}} \left[ \max_{(y,r) \in B_\tau^S \times [0,t_0(\tau)]} \left(\frac{\|g_r\|_{g_0}}{\sqrt{m-1}}\right)^{\frac{m-1}{2}} \right]^{-1},
$$
and it is small enough in such a way that $g_r(y) \ge \frac{1}{4}g_0(y)$ for each $y \in S$ and $r < 2 t_0(\tau)$.
\end{itemize}
Let $\varphi \in C^\infty(\overline{\mathscr{D}})$ satisfy 
$$
0 \le \varphi \le 1, \qquad \varphi \equiv 1 \ \text{on } \, \{r \le t_0(\rho(y))\}, \qquad \varphi \equiv 0 \ \text{ on } \, \left\{ r \ge \frac{3}{2}t_0(\rho(y))\right\},
$$
and choose a smooth function $\eta : S \times [0,1] \ra (0,1]$ such that 
$$
\left\{ \begin{array}{ll}
\eta(y,r) = 1 & \quad \text{if } \, r < t_0(\rho(y)) \ \text{ or } \, r > 1-t_0(\rho(y)), \\[0.2cm]
0 < \eta(y,r) \le t_0(\rho(y))^{2}
& \quad \text{if } \, r \in [2t_0(\rho(y)), 1- 2t_0(\rho(y))]. 
\end{array}\right.
$$
We consider the following metric on $P = S \times [0,1]$:
$$
a \ \ \doteq \ \ \eta(y,r)^2\di r^2 + \Big[ \varphi(y,r) g_r(y) + \big(1-\varphi(y,r)\big)g_0(y)\Big].
$$
Then, $a$ coincides with $g$ in a neighbourhood of $S \times\{0\}$ and with the product metric on $S \times \R$ in a neighbourhood of $S \times \{1\}$. We claim that $(P,a)$ has finite volume. Indeed, by the inequality between the geometric and quadratic means, the volume density of $a$ satisfies 
$$
\begin{array}{lcl}
\disp \mathrm{det}_{g_0}(a)^{\frac{1}{m-1}} & \le & \disp \eta^{\frac{1}{m-1}} \Big[ \mathrm{det}_{g_0} (\varphi g_r + (1-\varphi)g_0) \Big]^{\frac{1}{m-1}} \le \eta^{\frac{1}{m-1}} \frac{\|\varphi g_r + (1-\varphi) g_0\|_{g_0}}{\sqrt{m-1}} \\[0.4cm]
& \le & \disp \eta^{\frac{1}{m-1}} \frac{\max\{ \|g_r\|_{g_0}, \|g_0\|_{g_0} \}}{\sqrt{m-1}} = \eta^{\frac{1}{m-1}}  \max \left\{ \frac{\|g_r\|_{g_0}}{\sqrt{m-1}}, 1 \right\}  
\end{array}
$$
and thus, letting $P_R = B_R^S \times [0,1]$ and noticing that $\int_0^1 \sqrt{\eta(y,r)} \di r \le 5 
t_0(\rho(y))$ we use the coarea formula to compute
$$
\begin{array}{lcl}
\disp |P_R|_g & \le & \disp \int_0^R \di \tau \left[\int_{y \in \partial B_\tau^S} \int_{r \in [0, 1]} \sqrt{\mathrm{det}_{g_0}(a)} \di y \di r \right] \\[0.4cm]
& \le & \disp \int_0^R \di \tau \left\{ \max_{(y,r) \in B_\tau^S \times [0,t_0(\tau)]} \left(\frac{\|g_r\|_{g_0}}{\sqrt{m-1}}\right)^{\frac{m-1}{2}} \int_{y \in \partial B_\tau^S}\int_0^1 \sqrt{\eta(y,r)} \di r \di y \right\} 
 \\[0.4cm]
& \le & \disp \int_0^R \frac{5	\di \tau}{1+\tau^2} \le \frac{5\pi}{2}.
\end{array}
$$
For $j \in \{1,2\}$, let $\Sigma_j', P_j$ be two isometric copies of, respectively, $\Sigma'$ and $P$. The metric $g$ on $\overline{\Sigma_j'}$ matches smoothly with $a$ on the union $N_j = \Sigma_j' \cup P_j$ along $\partial \Sigma_j' = S \times \{0\}$, generating a metric $h_j$. Define $(N,h)$ to be the gluing of $(N_1,h_1)$ with $(N_2,h_2)$ along their totally geodesic boundaries. In view of $(\wp)$, observe that the inclusion $S \hookrightarrow \partial \Sigma_j' \subset N$ is an embedding with closed image.\par
We shall prove that $(N,h)$ is complete. To this aim, consider $P_1\cup P_2 \subset N$ as being parametrized by $S\times[0,2]$, with $P_1 = S \times [0,1]$ and $P_2 = S \times [1,2]$, and let $\pi : P_1 \cup P_2 \to S$ be the canonical projection onto the first factor. Also let $f : \overline{\Sigma'_1} \to \overline{\Sigma'_2}$ be the Riemannian isometry sending each point of $\overline{\Sigma'_1}$ to its copy in $\overline{\Sigma'_2}$, and define the retraction $F : N \to \overline{\Sigma'_1}$ by setting
	\[
	F(x) = \begin{cases}
		x & \text{if } \, x \in \Sigma'_1, \\
		\pi(x) & \text{if } \, x \in P_1 \cup P_2, \\
		f^{-1}(x) & \text{if } \, x \in \Sigma'_2.
	\end{cases}
	\]
First, observe that $F$ is proper: indeed, for every compact set $K\subseteq \overline{\Sigma_1'}$,
	\[
	F^{-1}(K) = K \cup \Big( (K\cap S)\times[0,2] \Big) \cup f(K)
	\]
is compact, being a finite union of compact sets. We also claim that $F$ is $2$-Lipschitz. To this aim, let $x,y\in N$ and $\eps>0$. We show that $\dist_g'(F(x),F(y)) < 2\dist_h(x,y) + 2\eps$. Let $\gamma : [0,T] \to N$ be a curve joining $x$ and $y$ and such that $\ell_h(\gamma) < \dist_h(x,y) + \eps$. By the transversality theorem, we can assume that $\gamma$ is transversal to $\partial(P_1 \cup P_2)$, so in particular $F \circ \gamma : I \to \overline{\Sigma_1'}$ is a piecewise smooth curve joining $F(x)$ and $F(y)$ and there exist $0 = s_0 < s_1 < s_2 < \dots < s_k = T$ such that, letting $I_j = (s_j,s_{j+1})$, for each $0 \leq j \leq k-1$
\[
\gamma(I_j) \subseteq \mathrm{Int}(P_1 \cup P_2) \quad \text{ or } \quad \gamma(I_j) \subseteq \Sigma_1' \cup \Sigma_2'
\]
and for each $0 \leq j \leq k-2$ the images $\gamma(I_j)$, $\gamma(I_{j+1})$ belong to distinct components of $N\setminus\partial(P_1 \cup P_2)$.
If $\gamma(I_j) \subseteq \Sigma_1' \cup \Sigma_2'$ then
\[
\ell_g((F\circ\gamma)_{|I_j}) = \ell_h(\gamma_{|I_j}).
\]
On the other hand, assume that $\gamma(I_j) \subseteq \mathrm{Int}(P_1 \cup P_2)$. Because of $(p_2)$ in the definition of $t_0$,
	\[
	h = a \ge \eta(y,r)^2\di r^2 + \frac{1}{4}g_0(y) \qquad \text{on } \, P_1 \cup P_2, 
	\]
	so for every tangent vector $V \in T(P_1\cup P_2)$ we have
	\[
	g(\pi_\ast V,\pi_\ast V) = g_0(\pi_\ast V,\pi_\ast V) \leq 4 h(V,V),
	\]
	hence
	\[
	\ell_g((F\circ\gamma)_{|I_j}) \leq 2 \ell_h(\gamma_{|I_j}).
	\]
Summarizing,
	\[
	\dist_g'(F(x),F(y)) \leq \ell_g(F\circ\gamma) \leq 2 \ell_h(\gamma) < 2 \dist_h(x,y) + 2\eps
	\]
	as claimed. We prove that $(N,h)$ is complete. Let $\{x_k\}$ be a Cauchy sequence in $(N,\dist_h)$. Then, $\{F(x_k)\}$ is a Cauchy sequence in $(\overline{\Sigma_1'},\dist_g')$ and therefore, by $(\wp)$, it converges to some $y \in \overline{\Sigma_1'}$. The properness of $F$ implies that $\{x_k\}$ has a limit point in $N$, hence it converges.
	
	We next observe that 
	\begin{equation}\label{eq_claim_dist}
		\frac{1}{2} \mathrm{dist}_g \le \frac{1}{2} \mathrm{dist}'_g \le \mathrm{dist}_h \le \mathrm{dist}'_g \qquad \text{on } \, \Sigma_1' \times \Sigma_1'.
	\end{equation}
	Let $x,y \in \Sigma_1'$. The first and last inequalities are obvious since $(N,h)$ and $(\Sigma, g)$ contain more curves joining $x$ to $y$ than $(\overline{\Sigma'},g)$, while the middle inequality is a consequence of $2$-Lipschitzianity of $F$, together with $F_{|\Sigma_1'} = \mathrm{id}_{\Sigma_1'}$.
	%we pick a unit speed, minimizing geodesic $\gamma : [0,T] \ra N$ from $x$ to $y$, we split $[0,T]$ into the sets $A,A'$ as in \eqref{def_AAprimo} and, similarly to \eqref{eq_interest}, we compute 	 
	%\tcb{	\[
	%	\begin{array}{lcl}
	%		\mathrm{dist}_h(x,y) & = & \disp T \ge \frac{1}{2} \left[ \mathrm{dist}'_g(x,y_0) + \sum_k \mathrm{dist}'_g(y_k, y_{k+1}) + \mathrm{dist}'_g(y_\infty,y)\right] \\[0.4cm]
	%		& \ge & \disp \frac{1}{2} \mathrm{dist}'_g(x,y),
	%	\end{array}
	%	\]
	%}
	%where $y_\infty$ is the limit of $\{y_j\}$ (if $k_0= \infty$) or $y_{k_0}$ otherwise.
	\par
	To conclude, let $o \in \Sigma'$, denote by $o_j$ the corresponding copy on $\Sigma_j'$, let $B^h_r(o_1)$ be the geodesic ball of radius $r$ centered at $o_1$ in $(N,h)$, and let $R = \mathrm{dist}_h(o_1,o_2)$. Choose a shortest curve $\gamma: I \to \Sigma$ joining $o$ to $\partial \Sigma'$, in particular $\gamma(I) \subset \overline{\Sigma'}$, let $y \in S$ be the ending point of $\gamma$ and let $\gamma_j$ be the copy of $\gamma$ in $\overline{\Sigma_j'}$. We denote with $\tau : [0,2] \to P_1 \cup P_2$ the curve $t \mapsto (y,t)$, and observe that $\ell_h(\tau) \le 2$ because $\eta \le 1$. By our requirement on $R_0$,
	\[
	R \le \ell_h(\gamma_2^{-1} \ast \tau \ast \gamma_1) \le 2 \mathrm{dist}_g(o, \partial \Sigma') + 2 = R_0 \le r,
	\]
and thus, because of \eqref{eq_claim_dist},
	\[
	\begin{array}{lcl}
		\disp |B^h_r|_h & \le & \disp |B^h_r(o_1) \cap \Sigma'_1|_g + |B_r^h(o_1) \cap (P_1 \cup P_2)|_a + |B_r^h(o_1) \cap \Sigma'_2|_g \\[0.2cm]	
		& \le & \disp |B^h_r(o_1) \cap \Sigma'_1|_g + 2 |P|_a  + |B_{r+R}^h(o_2) \cap \Sigma'_2|_g \\[0.2cm]
		& \le & \disp 2 |B_{r+R}^h(o) \cap \Sigma'|_g + 5\pi \le 2 |B_{2r}^h(o) \cap \Sigma'|_g + 5\pi \\[0.2cm]
		& \le & \disp 2 |B_{4r}^g(o) \cap \Sigma'|_g + 5\pi,	
	\end{array}
	\]
	as required.
\end{proof}

\begin{lemma}\label{lem_2}
Let $\Omega \subset M$ be a connected open subset of a complete manifold $(M^m,\sigma)$, and let $u \in C^2(\Omega)$ solve $\MM[u] = H$ for some $H \in C(\Omega)$.\par
Then, given every connected subset $\Omega'$ with smooth boundary (in the sense of Remark \ref{rem_notbad}) whose closure in $M$ satisfies $\overline{\Omega'} \subset \Omega$, the graph $\Sigma'$ of $u$ over $\Omega'$ satisfies property $(\wp)$ in Lemma \ref{lem_stochest}. Furthermore, assume that
\begin{itemize}
	\item[-] $M$ satisfies
	\[ 
	\Ricc(Dr, Dr) \ge - \kappa^2(1+r^2) 
	\]
where $\kappa \in \R^+$, $r$ is the distance in $M$ from a fixed origin $\bar o \in \Omega'$, and the inequality holds outside of the cut-locus of $\bar o$; 
	\item[-] $H \in L^\infty(\Omega)$,
\end{itemize}
and that either of the following two properties holds:
	\begin{itemize}
	\item[$(i)$] $\Omega$ has locally Lipschitz boundary %finite perimeter
	and
$$
\liminf_{r \ra \infty} \frac{\log |\partial \Omega \cap B_r(\bar o)|_\sigma}{r^2} < \infty;
$$
	\item[$(ii)$] $u \in C(\overline \Omega)$ and is constant on $\partial \Omega$.
	\end{itemize}
Then, the manifold $(N,h)$ constructed in Lemma \ref{lem_stochest} is stochastically complete.
\end{lemma}	
	
\begin{proof}
Let $\pi : \Sigma \to \Omega$ the diffeomorphism given by the restriction of the projection $M \times \R \to M \times \{0\}$ to $\Sigma$. Observe that $\Sigma'$ is a subset with smooth boundary in $\Sigma$, and that its closure $\overline{\Sigma'}$ in $\Sigma$ coincides with $\pi^{-1}(\overline{\Omega'})$. To check that $(\wp)$ in Lemma \ref{lem_stochest} is satisfied, let $U \subset (\overline{\Sigma'},g)$ be bounded in the metric induced by $g$ on $\overline{\Sigma'}$. Since $\pi$ does not increase distances, $\pi(U)$ is bounded in $\overline{\Omega'}$ with the distance induced by $\sigma$, in particular, $\pi(U)$ is bounded in $M$. The completeness of $M$ ensures that $\overline{\pi(U)} \subseteq \overline{\Omega'}$ is compact in $M$, hence in $\overline{\Omega'}$. Pulling back with $\pi$, we conclude that $U$ has compact closure in $\overline{\Sigma'}$, as claimed (cf. Remark \ref{rem_comple}). Lemma \ref{lem_stochest} is therefore applicable. We then use a criterion in Section 4 of \cite{prsmemoirs}, according to which $(N,h)$ is stochastically complete whenever 
	\begin{equation}\label{eq_stoch}
	\liminf_{r \ra \infty} \frac{\log |B_r^h|_h}{r^2} < \infty,
	\end{equation}
and a calibration argument (together with Lemma \ref{lem_stochest}) to estimate $|B_r^h|_h$. Write $B_r$ instead of $B_r(\bar o)$ on $M$ for notational convenience, let $| \cdot |$ be the Hausdorff measure of the appropriate dimension on $M \times \R$ induced by the product metric, and denote by $\BB_r$ a geodesic ball in $M \times \R$ with center $o = (\bar o, u(\bar o))$. Fix a sequence $\{\Omega_j\}$ of open sets with smooth boundary satisfying $\overline{\Omega_j} \subset \Omega$, to be specified later, and let $\Sigma_j$ be the graph over $\Omega_j$. Consider the cylinder $C_r = (B_r \cap \Omega_j) \times I_r$, with $I_r = [u(\bar o) -r,u(\bar o)+ r]$. Note that 
$$ 
|\Sigma_j \cap B_r^g|_g \le |\Sigma_j \cap \BB_r| \le |\Sigma_j \cap C_r|,
$$
where recall that $B_r^g$ is a metric ball centered at $o$ in $(\Sigma,g)$. Extend the upward-pointing normal vector $\nu$ of $\Sigma$ in a parallel way along the $t$ direction to a field on the entire $\Omega \times \R$, and note that $\overline{\diver}(\nu) = -H \le \|H\|_\infty$. Let $U_r$ be the union of the connected components of $C_r \backslash \Sigma_j$ lying below $\Sigma_j$. By Sard's theorem for the distance function, we can choose $r$ in such a way that $B_r \cap \Omega_j$ has Lipschitz boundary for every $j$. Then, denoting by $\eta$ the exterior normal to $U_r$,
\[
\|H\|_\infty |B_r \times I_r| \ge \disp \|H\|_\infty |U_r| \ge \int_{U_r} \overline{\diver}(\nu) = \int_{\partial U_r} \langle \nu, \eta \rangle = |\Sigma_j \cap C_r| + \int_{\partial U_r \backslash \Sigma} \langle \nu, \eta \rangle.
\]
%	\[
%	\begin{array}{lcl}
%	\|H\|_\infty |B_r \times I_r| & \ge & \disp \|H\|_\infty |U_r| \ge \int_{U_r} \overline{\diver}(\nu) = \int_{\partial U_r} \langle \nu, \eta \rangle = |\Sigma_{\tcb{j}} \cap C_r| + \int_{\partial U_r \backslash \Sigma} \langle \nu, \eta \rangle \\[0.4cm]
%	& \ge & \disp |\Sigma_j \cap C_r| + \int_{(\partial \Omega_j \cap B_r) \times I_r} \langle \nu, \eta \rangle - |\partial(B_r \times I_r)|.
%	\end{array}
%	\]
The set $\partial U_r \setminus \Sigma$ consists of the union of the connected components of $(\partial \Omega_j \cap B_r) \times I_r$ lying below $\Sigma$ and of a subset of $\partial (B_r \times I_r)$. We split the integral of $\langle \nu, \eta \rangle$ over $\partial U_r \setminus \Sigma$ as the sum of the two integrals over these complementary subsets, namely 
\[
\begin{array}{lcl}
	A_1 & = & (\partial U_r \setminus \Sigma) \cap [(\partial \Omega_j \cap B_r) \times I_r], \\[0.2cm]
	A_2 & = & (\partial U_r \setminus \Sigma) \setminus A_1 \subseteq \partial (B_r \times I_r)
\end{array}
\]
and we estimate
\[
\left| \int_{A_2} \langle \nu, \eta \rangle \right| \leq \int_{A_2} |\langle \nu, \eta \rangle| \leq |A_2| \leq |\partial (B_r \times I_r)|
\]
to obtain
\[
\|H\|_\infty |B_r \times I_r| \ge \disp |\Sigma_j \cap C_r| + \int_{A_1} \langle \nu, \eta \rangle - |\partial(B_r \times I_r)|.
\]
Summing up, 
\[
|\Sigma_j \cap B_r^g|_g \le |\partial (B_r \times I_r)| + \|H\|_\infty |B_r \times I_r| - \int_{A_1} \langle \nu, \eta \rangle. 
\]
Using the volume comparison theorem, the lower bound on $\Ricc(Dr,Dr)$ guarantees that
	\[
	\left\{ \begin{array}{l}
	|\partial (B_r \times I_r)| = \disp 2r |\partial B_r|_\sigma + 	2|B_r|_\sigma \le 2r v(r) + 2V(r), \\[0.2cm]
	|B_r \times I_r| \le 2r V(r),
	\end{array}\right.
	\]
with $v(r),V(r)$ the volume of a sphere and ball of radius $r$ in a rotationally symmetric manifold of curvature $-\kappa^2(1+ r^2)$, which particular satisfy (cf. Chapter 2 in \cite{prs})
	\[
	v(r) + V(r) \le C_1 \exp\big\{ C_2 r^2\big\}, 
	\]
	for some constants $C_j = C_j(m,\kappa)>0$. We therefore obtain 
	\begin{equation}\label{calibration}
	|\Sigma_j \cap B_r^g|_g \le C_3(1+ \|H\|_\infty)r \exp\{C_2 r^2\} - \int_{A_1} \langle \nu, \eta \rangle.
	\end{equation}
\par
We now split the argument according to whether $(i)$ or $(ii)$ holds. If $(i)$ is in force, we choose $\{\Omega_j\}$ in such a way that $\overline{\Omega'} \subset \Omega_j$ and the perimeter measure of $\partial \Omega_j$ locally converges to that of $\partial \Omega$ (this is possible, see Remark \ref{rem_regu} below). By Cauchy-Schwarz inequality we estimate
$$
\left| \int_{A_1} \langle \nu, \eta \rangle \right| \leq |A_1| \leq |(\partial \Omega_j \cap B_r) \times I_r| = 2r|\partial \Omega_j \cap B_r|_\sigma,
$$
hence from \eqref{calibration} we deduce 
	\[
	\begin{array}{lcl}
	|\Sigma' \cap B_r^g|_g & \le & |\Sigma_j \cap B_r^g|_g \le C_3(1+ \|H\|_\infty)r \exp\{C_2 r^2\} + 2r|\partial \Omega_j \cap B_r|_\sigma \\[0.2cm]
	& \ra & \disp C_3(1+ \|H\|_\infty)r \exp\{C_2 r^2\} + 2r |\partial \Omega \cap B_r|_\sigma
	\end{array}
	\]
as $j \ra \infty$. By Lemma \ref{lem_stochest}, for large $r$
	\[
	|B_r^h|_h \le 16 r |\partial \Omega \cap B_{4r}|_\sigma + 5\pi + C_5(1+ \|H\|_\infty) r \exp\{ C_6 r^2\},
	\]
	whence
	\[
	\liminf_{r \ra \infty} \frac{\log |B_r^h|_h}{r^2} \le C_6 + \frac{1}{16} \liminf_{r \ra \infty} \frac{\log |\partial \Omega \cap B_r|_\sigma}{r^2} < \infty,
	\]
that was to be proved.\par
Next, we consider assumption $(ii)$, so without loss of generality assume $u=0$ on $\partial \Omega$ and that $u$ is non-constant. Suppose that $\{u>0\} \neq \emptyset$, and set $\Omega_j = \{ u > \eps_j\}$, with $\eps_j \downarrow 0$ chosen in such a way that $\Omega_j$ has smooth boundary. Notice that $\eta = - Du/|Du|$ on $\partial \Omega_j$, thus $\langle \nu, \eta\rangle \ge 0$ and \eqref{calibration} yields to 
	\begin{equation}\label{calibration2}
	|\Sigma_j \cap B_r^g|_g \le C_3(1+ \|H\|_\infty)r \exp\{C_2 r^2\}.
	\end{equation}
Letting $j \ra \infty$, we infer that the graph of $\{u>0\}$ in $\Sigma$ satisfies	
	\[
	|\{u> 0\} \cap B^g_r|_{g} \le C_3(1+ \|H\|_\infty)r\exp\{ C_2 r^2\}.
	\]
On $\{u<0\}$, we apply the above procedure to the graph of $-u$ and deduce 
	\[
	|\{u < 0\} \cap B^g_r|_{g} \le C_3(1+ \|H\|_\infty)r\exp\{ C_{2} r^2\}
	\]
(we here used that, when written via the map $\pi : \Sigma \to \Omega$, $B_r^g$ is the same ball both for the graph of $u$ and for the graph of $-u$).
	%Similarly, if $\{u<0\} \neq \emptyset$, we set $\Omega_j = \{u < -\eps_j\}$ and integrate the inequality $\overline{\diver}(\nu) \ge - \|H\|_\infty$ on the subset $\hat U_r \subset C_r$ lying above $\Sigma_j$. The boundary integral in \eqref{calibration} still appears with a negative sign, yielding to  
%	\[
%	|\{u < 0\} \cap B^g_r|_{g} \le C_3(1+ \|H\|_\infty)r\exp\{ C_2 r^2\}.
%	\]
On the other hand, by Stampacchia's theorem $Du_j = 0$ a.e. on $\{u_j = 0\}$, hence $g = \sigma$ a.e. on $\{u_j=0\}\cap \Omega$, that implies 
	\[
	|\{u=0\} \cap B_r^g|_g \le |\{u = 0\} \cap \Omega \cap C_r| \le |B_r|_\sigma \le C_1 \exp\{ C_2 r^2\}. 
	\]
Summarizing, 
	\[
	|\Sigma' \cap B_r^g|_g \le C_5(1+ \|H\|_\infty) r\exp\{ C_6 r^2\},
	\]
and we deduce as in $(i)$ the validity of \eqref{eq_stoch}.	
	\end{proof}

\begin{remark}\label{rem_regu}
\emph{The regularity of $\partial \Omega$ required in $(i)$ of Lemma \ref{lem_2} is only needed to guarantee the existence of smooth open sets $\{\Omega_j\}$, with $\overline{\Omega'} \subset \Omega_j \subset \Omega$, such that, on each compact set $K$, the $(m-1)$-dimensional Hausdorff measure $|\partial\Omega_j \cap K|$  converges to $|\partial\Omega \cap K|$. If $\partial \Omega$ is locally Lipschitz, to construct $\Omega_j$ one can apply in a local chart an interior approximation result for Lipschitz regular open subsets of $\R^m$, see for instance \cite{dok76}, and then a covering argument on $K$ together with a diagonal argument for $K$ exhausting $M$. In recent years, several authors investigated the existence of such a smooth interior approximation for less regular $\Omega$. In particular, we mention the following result from \cite{schmidt}: if $\Omega\subseteq \R^m$ is a bounded set with finite perimeter, whose perimeter measure equals the $(m-1)$-Hausdorff measure of the topological boundary $\partial \Omega$, then there exists a sequence of open smooth subsets $\Omega_j \subseteq \Omega$ such that $\partial\Omega_j \to \partial\Omega$ in the Hausdorff distance and $|\partial\Omega_j| \to |\partial\Omega|$ as $j\to\infty$.
}
\end{remark}

\begin{lemma}[\textbf{AK-duality}]\label{lem_varrhop}
Let $(N,h)$ be a stochastically complete manifold. Then, for every small $\eps>0$ and every $p \in N$, there exists a function $\varrho_p \in C^\infty(N)$ satisfying 
$$
\left\{\begin{array}{l}
\varrho_p \ge 0, \qquad \varrho_p(x) = 0 \ \text{ if and only if } \, x = p, \\[0.2cm]
\varrho_p(x) \ra \infty \ \text{ as } \, x \ra \infty, \\[0.2cm]
\Delta_h \varrho_p \le \eps - \|\nabla \varrho_p\|_h^2 \qquad \text{on } \, N.
\end{array}\right.
$$
\end{lemma}

\begin{proof}
Consider the ball $B_\eps$ centered at $p$ in the metric $h$, with $\eps$ small enough that $B_{4\eps}$ is regular. By the AK-duality \cite{marivaltorta,maripessoa,maripessoa_2}, since $(N,h)$ is stochastically complete we can find a function $w \in \mathrm{LSC}(\overline{N \backslash B_\eps})$ solving 
$$
\left\{ \begin{array}{l}
w > 0 \quad \text{on } \, N \backslash B_\eps,  \qquad w(x) \ra \infty \ \text{ as } \, x \ra \infty, \\[0.2cm]
\Delta_h w \le \eps w \qquad \text{on } \, N \backslash \overline{B}_\eps.
\end{array} \right.
$$
The construction in the above references makes use of solutions of obstacle problems, that are carefully stacked to produce $w$, and in fact, in our setting, $w$ can be proved to be locally Lipschitz. However, the linearity of $\Delta_h$ allows to considerably simplify the method: fix an exhaustion $\Omega_j \uparrow N$ by means of relatively compact, smooth open sets with $B_{4\eps} \Subset \Omega_1$, and for each $j$ let $u_j$ solve
$$
\left\{ \begin{array}{l}
\Delta_h u_j = \eps u_j \qquad \text{on } \, \Omega_j\backslash B_{\eps}, \\[0.2cm]
u_j = 0 \quad \text{on } \, \partial B_{\eps}, \qquad u_j = 1 \quad \text{on } \, \partial \Omega_j.
\end{array}\right.
$$
The maximum principle implies that $0 < u_j < 1$ on $\Omega_j \backslash B_{\eps}$, thus the extension of $u_j$ obtained by setting $u_j =1$ on $N \backslash \Omega_j$ is a locally Lipschitz, weak solution of $\Delta_h u_j \le \eps u_j$ on $N \backslash \overline{B}_{\eps}$. By elliptic estimates, $u_j \downarrow u$ locally smoothly on $N \backslash \overline{B}_{\eps}$, for some $0 \le u \in C^\infty(N \backslash \overline{B}_{\eps}) \cap C(N \backslash B_\eps)$ solving  
$$
\left\{ \begin{array}{l}
\Delta_h u = \eps u \qquad \text{on } \, N \backslash \overline{B}_{\eps}, \\[0.2cm]
u = 0 \quad \text{on } \, \partial B_{\eps}.
\end{array}\right.
$$
Extending $u$ with $0$ on $B_\eps$, we obtain $\Delta_h u \ge \eps u$ weakly on $N$, and since $(N,h)$ is stochastically complete, necessarily $u \equiv 0$. For each $k \in \mathbb{N}$, fix $j(k)$ such that $u_{j(k)} < 2^{-k}$ on $\Omega_k$ and $\{j(k)\}$ is increasing. Then, the series
$$
\sum_{k=1}^\infty u_{j(k)}
$$
converges locally uniformly to $w \in \lip_\loc(N \backslash \overline{B}_\eps)$ satisfying $w \ge n$ on $N \backslash \Omega_{j(n)}$. So, $w$ is an exhaustion and solves $\Delta_h w \le \eps w$ weakly on $N \backslash \overline{B}_{\eps}$, by stability of weak supersolutions.\par
The local Lipschitz regularity of $w$ allows to apply Greene-Wu approximation theorems \cite{GW}: consider $-2-w$, which satisfies $\Delta(-2-w) \ge -\eps w$ weakly on $N \backslash \overline{B}_{\eps}$. By \cite[Cor. 1 of Thm. 3.1]{GW} there exists $w_1 \in C^\infty(N \backslash \overline{B}_{\eps})$ satisfying
	\[
	|w_1 - (w+2)| < \frac{1}{2}, \qquad \Delta(-w_1) \ge -\eps w -\eps,
	\]
so in particular $w>3/2$ and $\Delta w_1 \le \eps(w+1) \le \eps w_1$ on $N \backslash \overline{B}_\eps$. Next, let $\psi$ be a cutoff of $B_{2\eps}$ in $B_{3\eps}$, and let $\bar w = \psi(1+r_p^2) + (1-\psi) w_1$, with $r_p$ the distance from $p$. Then, $\bar w$ is defined and smooth on $N$, $\bar w > 1$ on $N$ with equality at $p$, and $\Delta_h \bar w \le C$ on $B_{4\eps}$ for some constant $C>0$. Setting 
	\[
	z = \frac{\bar w + C/\eps}{1 + C/\eps},
	\]
we infer $\Delta_h z \le \eps z$ on $N$, $z \ge 1$ on $N$ with equality only at $p$, and $z(x) \ra \infty$ as $x \ra \infty$. To conclude, the function $\rho_p = \log z$ satisfies all of the desired properties.
\end{proof}

We are ready for the

\begin{proof}[Proof of Theorem \ref{teo_graph_soloricci_intro}]
Let 
	\begin{equation} \label{def_C}
		C > \kappa \sqrt{m-1},
	\end{equation}
and choose $\eps' \in (0,1)$ such that 
	\begin{equation}\label{def_epsprimo}
	C^2(1-2\eps') - (m-1)\kappa^2 > 0.
	\end{equation}	
Define for convenience 
\[
	z_0 \doteq We^{-Cu}.
\]
By contradiction, assume that for some $\tau>0$ the upper level set
	\begin{equation}\label{eq_nonmale}
	\hat \Omega_{\tau} \doteq \left\{ x \in \Omega \ : \ z_0(x) > \max \left\{ 1, \limsup_{y \ra \partial \Omega} \frac{W(y)}{e^{\kappa\sqrt{m-1}u(y)}} \right\} + \tau \right\} \neq \emptyset.
	\end{equation}
Let $\Omega_\tau$ be a connected component of $\hat \Omega_\tau$, and fix $p \in \Omega_\tau$. By Sard's theorem, we can assume that the gradient of $z_0$ does not vanish on $\partial \Omega_\tau$, so $\Omega_\tau \subset \Omega$ is an open subset with smooth boundary, whose closure $\overline{\Omega_\tau}$ in $M$ satisfies $\overline{\Omega_\tau} \subset \Omega$. Lemma \ref{lem_2} can therefore be applied to show that the graph $\overline{\Sigma_\tau}$ over $\overline{\Omega_\tau}$ is isometric to an open subset of a complete, connected, stochastically complete manifold $(N,h)$. We extend $W$ and $z_0 \doteq W e^{- C u}$ to smooth functions on $N \backslash \overline{\Sigma}_{\tau}$ satisfying
	\[
	W \ge 1, \qquad z_0 < \max \Big\{ 1, \limsup_{y \ra \partial \Omega} W(y) e^{-\kappa\sqrt{m-1}u(y)} \Big\} + \tau \qquad \text{on } \, N \backslash \overline{\Sigma}_\tau.
	\]
For $\eps >0$ to be specified later, let $\varrho_p \in C^\infty(N)$ be the exhaustion function constructed in Lemma \ref{lem_varrhop}, and for $\delta \in (0,1)$ consider
	\[	
	z = z_0 e^{-\eps \varrho_p} - \delta W \qquad \text{on } \, N.
	\]
Note that
	\[
	z = \eta W \qquad \text{on } \, \Sigma_\tau, \ \ \text{ with } \, \eta = e^{-Cu - \eps \varrho_p} - \delta.
	\]
	From $\varrho_p(p)=0$, we can choose $\delta$ small enough to satisfy
	\[
	p \in U_{\tau} \doteq \Big\{ x \in N \ : \ z > \max \Big\{ 1, \limsup_{y \ra \partial \Omega} W(y) e^{-\kappa\sqrt{m-1}u(y)} \Big\} + \tau \Big\}.
	\]
	Observe that, by construction and since $\varrho_p$ is an exhaustion, $U_\tau \subset \Sigma_\tau$ and it is relatively compact. The following computations are performed on $U_\tau$, where by construction $h=g$. Differentiating, 
	\[
	\begin{array}{rcl}
	\nabla z & = & \eta \nabla W + W \nabla \eta \\[0.3cm]
	\disp \Delta_g z & = & \disp \eta \Delta_g W + W \Delta_g \eta + 2 \langle \nabla \eta, \nabla W \rangle \\[0.4cm]
	& = &  \disp \eta \Delta_g W + W \Delta_g \eta - 2 \frac{\|\nabla W\|^2}{W} \eta + 2 \langle \frac{\nabla W}{W}, \nabla z \rangle. \\[0.4cm]
	\end{array}
	\]
From \eqref{jacobi_forW}	
	\[
	\Delta_g W \ge - (m-1)\kappa^2 \frac{W^2-1}{W^2} W + 2 \frac{\|\nabla W\|^2}{W}.
	\]
	Hence, since $z>0$ on $U_{\tau}$,
	\begin{equation}\label{ineq_bella}
	\disp \Delta_g z - 2 \langle \frac{\nabla W}{W}, \nabla z \rangle \ge \disp W \left[ \Delta_g \eta - (m-1)\kappa^2 \dfrac{W^2-1}{W^2} \eta \right] \qquad \text{on } \, U_\tau.
	\end{equation}
	By the definition of $\eta$, Young's inequality and $\|\nabla u\|^2 = (W^2-1)/W^2$, we get 
	\[
	\begin{array}{lcl}
	\frac{1}{\eta + \delta} \Delta_g \eta & = & \disp \| C \nabla u + \eps \nabla \varrho_p\|^2 - \eps \Delta_g \varrho_p \\[0.2cm]
	& \ge & C^2(1-\eps')\|\nabla u\|^2 - \eps^2\left(\frac{1}{\eps'}-1\right)\|\nabla \varrho_p\|^2 - \eps (\eps - \|\nabla \varrho_p\|^2) \\[0.2cm]
	& \ge & \disp C^2(1-\eps')\frac{W^2-1}{W^2} - \eps^2 \ge \disp C^2(1- 2\eps')\frac{W^2-1}{W^2}
	\end{array}
	\]
	where, in the second inequality, we have chosen $\eps = \eps(\eps')$ small enough in such a way that the coefficient of $\|\nabla \varrho_p\|^2$ be positive, and in the last one we possibly further reduce $\eps$ to satisfy
	\[
	\eps^2 \le C^2\eps' \frac{(1+\tau)^2-1}{(1+\tau)^2} \le C^2 \eps' \frac{W^2-1}{W^2} \qquad \text{on } \, U_{\tau}
	\]
	(recall that $W > z > 1+\tau$ on $U_{\tau}$). Hence, in view of \eqref{def_epsprimo},
	\begin{equation}\label{ineq_bella_2}
	\begin{array}{lcl}
	\disp \frac{1}{W} \left[\Delta_g z - 2 \langle \frac{\nabla W}{W}, \nabla z \rangle\right] & \ge & \disp \Delta_g \eta - (m-1)\kappa^2 \dfrac{W^2-1}{W^2} \eta \\[0.2cm]
	& \ge & \disp C^2(1-2\eps')\dfrac{W^2-1}{W^2}(\eta + \delta) - (m-1)\kappa^2 \dfrac{W^2-1}{W^2} \eta \\[0.4cm]
	& \ge & \disp \left[C^2(1-2\eps') - (m-1)\kappa^2\right] \frac{W^2-1}{W^2} (\eta+ \delta) \\[0.4cm]
	& \ge & \disp \left[C^2(1-2\eps') - (m-1)\kappa^2\right] \frac{(1+\tau)^2-1}{(1+\tau)^2}\delta > 0.
	\end{array}
	\end{equation}	
However, since $U_\tau$ is a relatively compact upper level set of $z$, evaluating \eqref{ineq_bella_2} at a maximum point of $z$ we reach the desired contradiction.\par
Concluding, for each $\tau$ the set $\Omega_\tau$ is empty, thus 
$$
\frac{W(x)}{e^{C u(x)}} \le \max \left\{ 1, \limsup_{y \ra \partial \Omega} \frac{W(y)}{e^{\kappa\sqrt{m-1}u(y)}} \right\} \qquad \text{for each } \, x \in \Omega,
$$
and the thesis follows by letting $C \downarrow \kappa \sqrt{m-1}$.\\[0.2cm]
Suppose now that $\Ricc \ge - (m-1)\kappa^2$ and that $u : M \ra \R^+$ is an entire, positive solution of $\MM[u]=0$. By the above, 
	\[
	\sqrt{1+ |Du|^2} \le e^{\kappa \sqrt{m-1} u} \qquad \text{on } \, M.
	\]
A posteriori, the same computation leading to \eqref{ineq_bella} can be performed with $\eps = \delta = 0$ and $C = \kappa \sqrt{m-1}$, and guarantees that 
	\[
	z = We^{-\kappa \sqrt{m-1} u}
	\]
satisfies 
	\[	
	\Delta_g z - 2 \langle \frac{\nabla W}{W}, \nabla z \rangle \ge - W \left[ \Delta_g(e^{-\kappa \sqrt{m-1} u}) - (m-1) \kappa^2 \frac{W^2-1}{W^2} e^{-\kappa \sqrt{m-1} u} \right] = 0.
	\]
Thus, if $z$ has a maximum point, then $z$ is a positive constant $c>0$. In particular, 
	\[
	|Du| = c \Big\{e^{2\kappa \sqrt{m-1} u} - 1 \Big\}^{\frac{1}{2}} \qquad \text{on } \, M
	\]
and 
	\[
	\Ricc(Du,Du) = -(m-1)\kappa^2 |Du|^2, \qquad \|\II\|^2 \equiv 0 \quad \text{on } \, M.
	\]	 
Therefore, by the expression of $\|\II\|^2$ we deduce $D^2 u \equiv 0$, hence $Du$ is parallel. If $u$ is non-constant, the base manifold $M$ would split as a Riemannian product $M_0 \times \R$ along the flow lines of $Du$, and $u$ would be a linear function of the $\R$-coordinate. This contradicts the fact that $u$ is non-constant and bounded from below. Hence, $u$ is constant, so equality in \eqref{eq_entire} forces $\kappa = 0$.
\end{proof}

\section{Remarks: harmonic functions and minimal graphs}

It is instructive to compare the proof of Theorem \ref{teo_graph_soloricci_intro} to that of the gradient estimate for entire harmonic functions on $M$. Suppose that $M$ satisfies $\Ricc \ge - (m-1)\kappa^2$ for some constant $\kappa \ge 0$, and let $u >0$ solve $\Delta u = 0$ on $M$. By \cite{yau,chengyau} and the subsequent refinement in \cite[Lem. 2.1]{liwang}, the following sharp inequality holds:
	\begin{equation}\label{chengyau}
	|D \log u| \le (m-1)\kappa \qquad \text{on } \, M.
	\end{equation}
In \cite{chengyau,liwang}, the authors achieve the goal by passing to the limit, as $R \ra \infty$, localized estimates for $|D\log u|$ on balls $B_R$. On the other hand, one can also prove the bound without the need to localize, in the spirit of \cite{yau}. The idea can briefly be  summarized as follows: first, by Bochner's formula 
	\begin{equation}\label{bochner}
	\frac{1}{2} \Delta |Du|^2 \ge |D^2 u|^2 - (m-1)\kappa^2 |Du|^2,
	\end{equation}
	and because of the refined Kato inequality
	\begin{equation}\label{kato_harm}
	|D^2 u|^2 \ge \frac{m}{m-1} \big|D|Du|\big|^2,
	\end{equation}
the function $z = |D\log u|$ turns out to solve
	\begin{equation}\label{eq_chengyau}
	\Delta z \ge \disp \frac{1}{m-1} \frac{|Dz|^2}{z} - (m-1)\kappa^2 z - \frac{2(m-2)}{m-1} \sigma(D \log u, D z) + \frac{1}{m-1} z^3
	\end{equation}
on $M$. Even though $z$ is not bounded a priori, because of the appearance of the coercive term $z^3$ (that precisely depends on the refined Kato inequality), the function 
	\[
	f(t) = \frac{1}{m-1} t^3 - (m-1)\kappa^2 t
	\]
satisfies the Keller-Osserman condition 
	\begin{equation}\label{KO}
	\int^\infty \frac{\di s}{\sqrt{F(s)}} < \infty, 			
	\end{equation}
where we set	
	\[
	F(s) = \int_0^s f(t) \di t
	\]
and note that $F>0$ for large $s$. Condition \eqref{KO} originated from the works of Keller \cite{keller} and Osserman \cite{osserman}, and by \cite{motomiya,prsmemoirs} it guarantees that $\sup_M z < \infty$ (in fact, we are neglecting the product term $\sigma(D \log u, D z)$, but this term can be suitably absorbed and is not relevant for the present discussion). Once $z$ is bounded, the Omori-Yau maximum principle (cf. \cite{prsmemoirs}) implies the inequality $f(\sup_M z) \le 0$, and consequently, \eqref{chengyau}.\par
As a matter of fact, once \eqref{eq_chengyau} is given,  the full strength of the lower bound $\Ricc \ge - (m-1)\kappa^2$ is not needed anymore to obtain 
	\begin{equation}\label{eq_desired_cy}
	\sup_M z < \infty, \qquad f\big(\sup_M z\big) \le 0.
	\end{equation}
Indeed, to reach the goal it would be sufficient that
	\begin{equation}\label{eq_uppervolume}
	\liminf_{r \ra \infty} \frac{\log |B_r|}{r^2} = 0,
	\end{equation}
a condition that is clearly implied by $\Ricc \ge -(m-1)\kappa^2$ via Bishop-Gromov theorem. The sufficiency of \eqref{eq_uppervolume}, although not explicitly remarked, can be deduced from the proof of an estimate analogous to \eqref{chengyau} in the more general case of $p$-harmonic functions, $p \in (1,\infty)$, recently appeared in \cite{sungwang} and Section 2 of \cite{maririgolisetti_mono}. There, the authors proved the sharp inequality
	\[
	|D \log u| \le \frac{m-1}{p-1} \kappa \qquad \text{on } \, M
	\]
for a positive solution of $\Delta_p u = 0$ by means of integral estimates. 

\begin{remark}
\emph{Note that the quadratic exponential upper bound for the volume growth in \eqref{eq_uppervolume} is borderline for the stochastic completeness of $M$, and this is not by chance. It is a sharp threshold to guarantee any of the following properties:
	\begin{itemize}
	\item[-] that the Keller-Osserman condition \eqref{KO} forces every solution of $\Delta z \ge f(z)$ to be bounded from above;
	\item[-] that any such solution of $\Delta z \ge f(z)$ with $\sup_M u < \infty$ satisfies $f(\sup_M u) \le 0$. By \cite{prsmemoirs}, this is equivalent to the stochastic completeness of $M$.
	\end{itemize}
However, the picture is subtle and a more precise account goes beyond the scope of the present work: the reader who is interested in the relations between Keller-Osserman conditions and maximum principles at infinity, is referred to \cite{bmpr} for a detailed discussion.
}
\end{remark}
In the attempt to adapt the argument to solutions of the minimal surface equation, clearly replacing the Bochner inequality \eqref{bochner} with its equivalent form given by the Jacobi equation \eqref{jacobi_forW}, the first main difficulty comes from the refined Kato inequality: using the harmonicity of $u$ on $\Sigma$ and \eqref{nablau2},
\begin{equation}\label{Kato}
\|\II\|^2 = W^2 \|\nabla^2 u\|^2 \ge W^2 \frac{m}{m-1} \big\| \nabla \|\nabla u\|\big\|^2 = \frac{m}{m-1} \frac{\|\nabla W\|^2}{W^2(W^2-1)}
\end{equation}
where $|Du|>0$. Differently from \eqref{kato_harm}, however, the helpful lower bound provided by \eqref{Kato} is lower order for large $W$, and seems too small to help in  \eqref{jacobi_forW}. Indeed, we are still unable to construct, out of $z_0 = We^{-Cu}$, a function $z$ solving a differential inequality that parallels \eqref{eq_chengyau} and contains a potential term $f(z)$ matching a Keller-Osserman condition. For this reason, we had to exploit the AK-duality to be able to localize the problem. The fact that estimates via localization on $\Sigma$ still provide a bound for $W$ indicates that some sort of coercivity might still appear in the equation satisfied by suitably chosen $z$. We intend to pursue this investigation further.

\section{Bernstein and half-space properties: proof of Theorem \ref{mainteo_tiporss}}\label{sec_Thm2}

\begin{proof}
The Bernstein property for entire minimal graphs on complete manifolds with $\Ricc \ge 0$ is an immediate consequence of Theorem \ref{teo_graph_soloricci_intro} applied with $\kappa = 0$.\par 
Regarding the half-space property, let $M$ be a complete parabolic manifold with $\Ricc \ge -(m-1)\kappa^2$ for some $\kappa>0$, and let $\varphi : S^m \ra M \times \R$ be a properly immersed minimal hypersurface contained in a closed half-space, that we can assume to be $M \times \R^+_0$. We follow the strategy in \cite{rosenbergschulzespruck}. Arguing by contradiction, we suppose that $\varphi(S)$ is not a slice. Up to downward translating $\varphi(S)$, we can assume that
\begin{equation}\label{inf_eq}
\inf\Big\{ t \in \R^+_0 : (x,t) \in \varphi(S) \; \text{for some } x \in M \Big\} = 0
\end{equation}
and by the tangency principle we have $\varphi(S) \cap \{t=0\} = \emptyset$. From this fact and properness of $\varphi$ we deduce the existence of a positive solution $u$ of $\MM[u]=0$, with bounded gradient, defined on the complement of a small ball $B_\eps \subset M$ and whose graph $\Sigma$ lies below $\varphi(S)$. Then, we prove that such $u$ must in fact be a positive constant, contradicting \eqref{inf_eq}. However, here the construction of $u$ differs from that carried out in \cite{rosenbergschulzespruck}, as we use Perron's method instead of an exhaustion argument on compacts $\Omega_j$ together with the continuity method. This approach, that we learnt from \cite{mazetwande,rite}, allows us to avoid the need of local estimates, necessary to pass to the limit the graphs in $\Omega_j \backslash B_\eps$, and directly guarantees the existence of a complete minimal graph to which Theorem \ref{teo_graph_soloricci_intro} will be applied.\par
Define
$$
w: M \ra \R^+_0, \qquad w(x) = \inf\Big\{ t \in \R^+_0 :  (x,t) \in \varphi(S) \Big\}. 
$$
As previously said, up to downward translating $\varphi(S)$ we can suppose that
$$
\inf w = 0.
$$
Observe that $w$ is lower-semicontinuous: indeed, let $y_i \ra y$ in $M$ and satisfying $w(y_i) \ra \liminf_{p \ra y} w(p)$. Without loss of generality, assume $w(y_i) < \infty$. Fix $t_i >0$ such that $(y_i, t_i) \in \varphi(S)$ and $t_i < w(y_i) + i^{-1}$. The properness of $\varphi$ ensures that $(y_i,t_i) \ra (y, t) \in \varphi(S)$ up to a subsequence, thus
$$
w(y) \le t = \lim_i w(y_i) = \liminf_{p \ra y} w(p). 
$$
%Also, up to a translation we can suppose that $\inf w = 0$.
\par
Fix $o\in M$ such that $w(o) < \infty$. Suppose, by contradiction, that $\varphi(S)$ is not a slice. Then $(o,0)\not\in\varphi(S)$ by the tangency principle and since $\varphi$ is proper, hence $w(o)>0$. Fix $0 < \delta < w(o)$. By lower-semicontinuity of $w$, there exists $\eps>0$ such that $w > \delta$ on $B_{2\eps} = B_{2\eps}(o)$, so that $B_{2\eps}\times[0,\delta]\cap \varphi(S)=\emptyset$. Hereafter, if not otherwise stated, geodesic balls in $M$ are thought to be centered at $o$. As in \cite{rosenbergschulzespruck}, for sufficiently small $\delta$ there exists $z \in C^2(\overline{B}_{2\eps}\backslash B_\eps)$ radially symmetric with respect to $o$, solving
\begin{equation}\label{eq_subsol}
\left\{ \begin{array}{l}
\MM[z] \ge 0 \qquad \text{on } \, B_{2\eps}\backslash \overline{B}_\eps, \\[0.2cm]
0 < z < \delta \qquad \text{on } \, B_{2\eps}\backslash \overline{B}_\eps \\[0.2cm]
z= \delta \quad \text{on } \, \partial B_{\eps}, \qquad z= 0 \quad \text{on } \, \partial B_{2\eps}
\end{array}\right.
\end{equation}
and satisfying
$$
	\sup_{\partial B_\eps}|Dz| < \infty.
$$
Extending $z$ by setting $z=0$ outside of $B_{2\eps}$ we still obtain a subsolution for the minimal surface equation, in the viscosity sense, on the entire $M \backslash B_\eps$. Consider the following modified Perron's class
$$
\begin{array}{ccl}
\FF & = & \disp \Big\{ u : M \backslash B_\eps \ra \R \ \ \text{ continuous such that} \ \ u \le \min\{\delta, w(x) \}  \\[0.4cm]
& & \disp \text{and} \ \ \MM[u] \ge 0 \ \ \text{ in the viscosity sense.} \Big\}.
\end{array}
$$
First, $\FF \neq \emptyset$ because it contains $z$, and we can therefore define Perron's envelope
$$
u(x) = \sup \big\{ v(x) \ \ : \ \ v \in \FF \big\},
$$
that is lower-semicontinuous and satisfies $u \le \delta$. %Furthermore, from $z \le u \le \delta$ we deduce that $u(x) \ra \delta$ as $x \ra \partial \Omega$.
We claim that $\MM[u]=0$ on $M \backslash \overline{B}_{\eps}$. Fix $x \in M \backslash \overline{B}_\eps$, and choose a sequence $\{v_j\} \subset \FF$ with $v_j(x) \ra u(x)$. Up to replacing $v_j$ with $\max\{v_1, \ldots, v_j\} \in \FF$, we can assume that $v_j(x) \uparrow u(x)$. Let $\bar \eps$ be small enough so that the normalized mean curvature $H$ of $\partial B_{\bar \eps}(x)$ in the inward direction is positive and large enough to satisfy the condition in \cite[Thm. 1]{dhd}. 
%$$
%\inf_{B_{\bar \eps}(x)} \Ricc > - \frac{(m-1)^2}{m} H^2.
%$$
Then, there exists a solution of the Dirichlet problem
$$
\left\{ 
\begin{array}{ll}
\MM[v'_j] = 0 & \quad \text{on } \, B_{\bar \eps}(x), \\[0.2cm]
v'_j = v_j  & \quad \text{on } \, \partial B_{\bar \eps}(x),
\end{array}
\right.
$$
and by comparison $v'_j \ge v_j$. Consider the minimal replacement $\harm(v_j)$ defined by 
$$
\harm(v_j) = \left\{ \begin{array}{ll}
v'_j & \quad \text{on } \, B_{\bar \eps}(x), \\[0.2cm]
v_j & \quad \text{on } \, M \backslash B_{\bar \eps}(x). 
\end{array}\right.
$$
We claim that $\harm(v_j) \in \FF$. First, $v_j \le \harm(v_j) \le \delta$ by comparison, and thus $\MM[\harm(v_j)] \ge 0$ in the viscosity sense. We are left to show that $\harm(v_j) \le w$ on $M \backslash B_{\bar \eps}$. Suppose by contradiction that $\max\{ \harm(v_j)(p) - w(p) : p \in \overline{B_{\bar \eps}(x)}\} >0$. Note that the maximum exists since $\harm(v_j)-w$ is upper-semicontinuous, and that it is attained at a point interior to $B_{\bar \eps}(x)$, since $\harm(v_j) = v_j \le w$ on $\partial B_{\bar \eps}(x)$. We can therefore translate the graph of $\harm(v_j)$ downwards up to a first interior touching point with $\varphi(S)$, yielding a contradiction by the tangency principle.\par 
Summarizing, from $\harm(v_j) \in \FF$ we deduce that $\harm(v_j)(x) \ra u(x)$. Elliptic estimates then yield the local smooth subconvergence of $\harm(v_j)$ to some solution $v$ of $\MM[v]=0$ on $B_{\bar \eps}(x)$, that also satisfies $v \le u$ with the equality at $x$. We claim that $u \equiv v$ on $B_{\bar \eps}(x)$. To this end, suppose that $u(p) > v(p)$ for some $p \in B_{\bar \eps}(x)$, consider $s \in \FF$ such that $s(p) > v(p)$ and the increasing sequence $\bar v_j = \max\{ v_j, s\}$. Let $\harm(\bar v_j)$ be the minimal replacement of $\bar v_j$. Again elliptic estimates guarantee that $\harm(\bar v_j) \ra \bar v$ on $B_{\bar \eps}(x)$ locally smoothly. By construction, $\bar v \ge v$ on $B_{\bar \eps}(x)$, with strict inequality at $p$ but with equality at $x$. This contradicts the tangency principle.\par
%Having proved that $u \equiv v$ on $B_{\bar \eps}(x)$ we immediately deduce that
By construction, $0 \le$ $u \le \min\{\delta,w\}$, in particular, $\inf_{M \setminus B_\eps} u = 0$. We now conclude the proof by showing that $u \equiv \delta$. 
%To conclude, it is enough to show that $u \equiv \delta$.
Note that $u$ is regular up to $\partial B_\eps$, and in particular, from $z \le u \le \delta$ ($z$ in \eqref{eq_subsol}), we deduce $u= \delta$, $|Du| \le |Dz|$ on $\partial B_\eps$. We can apply the gradient estimate in Theorem \ref{teo_graph_soloricci_intro} to deduce $|Du| \le C$ on $M \setminus \overline{B_\eps}$, for some constant $C$. We can therefore conclude exactly as in \cite{rosenbergschulzespruck}: namely, if $(\Sigma,g)$ is the graph of $u$, from the gradient estimate we deduce
$$
\int_{\Sigma} \|\nabla \phi\|^2 \di x_g \le C' \int_M |D \phi|^2 \di x \qquad \text{for each } \, \phi \in \lip_c(M \backslash \overline{B}_\eps), 
$$
for some constant $C'>0$, thus $\Sigma$ is a parabolic manifold with boundary. Consequently, the harmonic function $u$ on $\Sigma$ is constantly equal to $\delta$, as desired.
\end{proof}

\begin{proof}[Proof of Theorem \ref{teo_tkachev}]
Because of volume comparison, $|B_r| \le C r^m$ for some $C>0$ and thus, by Theorem 10.30 in \cite{bmpr}, $f(u(x)) = 0$ for every $x \in M$. The function $u$ is therefore a minimal graph, and since $f \not \equiv 0$ then $u$ is bounded on one side. By the Bernstein property, $u$ is constant.
\end{proof}




\end{document}